\documentclass[a4paper,11pt]{article}
\usepackage[utf8]{inputenc}
\usepackage[english]{babel}
\usepackage{amsthm}

\usepackage[utf8]{inputenc}
\usepackage{fontenc}
\usepackage{amsmath}
\usepackage{amsfonts}
\usepackage{amssymb}
\usepackage{graphicx}
\usepackage{amssymb}
\usepackage{hyperref}
\usepackage{enumitem}
\usepackage[dvipsnames]{xcolor}
\usepackage[left=2cm,right=2cm,top=2cm,bottom=4cm]{geometry}
\usepackage{tikz}

\newcommand{\ben}{\begin{enumerate}}
	\newcommand{\een}{\end{enumerate}}
\newtheorem{thm}{Theorem}
\newtheorem{cor}{Corollary}
\newtheorem{lem}{Lemma}

\newtheorem{dfn}{Definition}

\newtheorem{rmq}{Remark}

\newtheorem{expl}[thm]{Example}
\newtheorem{problem}{Problem}[section]
\newtheorem{theorem}{Theorem}[section]

\newtheorem{lemma}{Lemma}[section]

\newtheorem{proposition}{Proposition}[section]
\newtheorem{remark}{Remark}[section]

\newcommand{\R}{\mathbb{R}}
\newcommand{\N}{\mathbb{N}}

\newcommand{\beq}{\begin{eqnarray}}
\newcommand{\eeq}{\end{eqnarray}}
\newcommand{\bprop}{\begin{pro}}
	\newcommand{\eprop}{\end{pro}}
\newcommand{\blem}{\begin{lem}}
	\newcommand{\elem}{\end{lem}}
\newcommand{\bdfn}{\begin{dfn}}
	\newcommand{\edfn}{\end{dfn}}
\newcommand{\bcor}{\begin{cor}}
	\newcommand{\ecor}{\end{cor}}
\newcommand{\bthm}{\begin{thm}}
	\newcommand{\ethm}{\end{thm}}
\newcommand{\bex}{\begin{expl}}
	\newcommand{\eex}{\end{expl}}
\newcommand{\brmq}{\begin{rmq}}
	\newcommand{\ermq}{\end{rmq}}
\newcommand{\bitem}{\begin{itemize}}
	\newcommand{\eitem}{\end{itemize}}
\newcommand{\bproof}{\begin{proof}}
	\newcommand{\eproof}{\end{proof}}

\begin{document}

	\title{Event-triggering mechanism to damp the linear wave equation }
	\author{
Florent Koudohode$^1$\footnote{e-mail: {\tt fkoudohode@laas.fr}.},	Lucie Baudouin $^2$\footnote{e-mail: {\tt lbaudouin@laas.fr}.},		
	Sophie Tarbouriech $^3$\footnote{e-mail: {\tt tarbour@laas.fr }.}  }
		
	\smallskip
	
	\footnotetext[1]{LAAS-CNRS, Universit\'e de Toulouse, CNRS, Toulouse, France}
	\date{\today}
	\maketitle
	\abstract{
		This paper aims at proposing a sufficient matrix inequality condition to carry out the global exponential stability of the wave equation under an event-triggering mechanism that updates a  
		damping source term. The damping is distributed in the whole space but sampled in time. The well-posedness of the closed-loop event-triggered control system is shown. Furthermore, the avoidance of Zeno behavior is ensured provided that the initial data are more regular. The interest of the results is drawn through some numerical simulations.
	}
	
	
\section*{Introduction}
The wave equation arises in fluid dynamics, acoustics and electromagnetics and models the evolution and the propagation of wave's amplitude (water waves, sound waves, seismic waves or light waves). In two space dimensions, it can be a model to study the vibration of a stretched elastic membrane like the skin of a drum and in one space dimension, it is called the vibrating cord or string equation.
The stabilization and analysis of stability of the wave equation are well-studied in the literature.\,The multiplier method used by \cite{chen1979control} and \cite{lions1988controlabilite}, a micro-local analysis approach by \cite{lebeau1996equation} and a backstepping method by \cite{smyshlyaev2010boundary} are used to characterize the stability and prove some controllability and stabilization results of this equation. In the present article, we focus on the stabilization problem, and more especially on the digital implementation of the control law. A natural choice is the periodic update of the control. However an alternative way is to design a triggering strategy, which determines through the occurrence of some events  when the control needs to be updated. It allows thus efficient use of communication and computational resources: it is the event-based control strategy \cite{heemels2012introduction}.

Event-triggered control can be defined as controls updated aperiodically, only when some triggering conditions occur.  Systems with event-based sampling are much harder to analyze than systems with periodic sampling because the time-varying nature of the closed-loop system cannot be avoided. Many difficulties that arise in the context of event-based control are due to the introduction of discontinuities when updating the control. Several works have been developed in this area for finite-dimensional networked control systems: see for instance the seminal works \cite{aaarzen1999simple,aastrom1999comparison} or the most recent ones \cite{heemels2012introduction}, \cite{postoyan2014framework}\,(for nonlinear system), \cite{tanwani2016observer}\,(with observers), \cite{peralez2018event}\,(with high gain approach) and the references therein. As far as infinite dimensional systems are concerned, there exist few designs of event-based control strategies  in the literature. Nevertheless, in  \cite{yao2013resource,selivanov2016distributed,espitia2020event} and \cite{espitia2017stabilization,davo2018stability,baudouin2019event} event-based control strategies were considered for parabolic and hyperbolic Partial Differential Equation (PDE). 

Two approaches are considered in the event-triggered control framework: the emulation problem and the co-design problem. The emulation problem is addressed if only the event-triggering rules are designed (\cite{postoyan2014framework,espitia2017stabilization,espitia2020event}). The joint design of the control law and the event-triggering conditions is referred in the literature as co-design (as for example, in \cite{seuret2016lq,heemels2015periodic}). 
In the current paper we are concerned with the first case, that is the emulation context. More precisely, we focus on proposing an event-triggered mechanism ensuring the exponential stability of the linear wave equation by means of a classical damping term, here sampled in time.  	

A fundamental issue when dealing with event-triggered controllers is to avoid any situation where the mechanism could induce infinitely many updates of the control in a bounded time interval, corresponding to the occurence of a Zeno phenomenon.  
A solution proposed in the literature to avoid Zeno behavior for infinite dimensional systems (see for example, \cite{espitia2016event,espitia2017event,espitia2017stabilization,baudouin2019event,kang2020event} for examples) consists in adding to the event-triggering condition a term exponentially decreasing and depending on the initial condition of the natural energy of the system. 
However, although this additional term allows the avoidance of Zeno behavior, it is restrictive in practice since it depends on the initial energy and a maximal decay rate. 
Actually, in the current paper, we follow another route in order to guarantee the absence of Zeno phenomenon. 
A simpler event-triggering mechanism is considered here, using only the error variable between the value of the state at the last triggering instant and the current one. 
Using an adequate Lyapunov functional, related to the energy of the system, the exponential stability of the closed-loop system under event-triggered control is proven. 
Furthermore, the avoidance of Zeno behavior, is guaranteed by showing the absence of accumulation points in the sequence of time updates. The stability condition proposed is expressed under the form of matrix inequality. The feasibility of such a  inequality is also studied. The main results presented in this paper can be viewed as complementary results to those developed in \cite{terushkin2019sampled,baudouin2019event}.

The paper is organized as follows. Section\,\ref{problemformuation} describes the context and the PDE system that we are concerned with. We present in Section\,\ref{etmstrategy} the main results on the proposed event-triggering mechanism. Hence, the well-posedness of the associated closed-loop system, the exponential stability and the avoidance of the Zeno phenomenon are ensured.  Section \ref{simulation} provides a numerical
example to illustrate the effectiveness of the approach. Conclusions and perspectives are given in Section~\ref{conclusion}. 

\subsection*{Notation.}
The gradient and the Laplacian of the function $z$ are respectively denoted by   $\nabla z=\left(\partial_{x_1} z,\dots \partial_{x_N} z \right)$ 
and $\Delta z=\sum_{i=1}^{N} \partial^2_ {x_i}z$
where $\partial_{x_i} z=\frac{\partial z}{\partial x_i}.$  
Given an open set $\Omega\subset \R^N ,$  $L^2(\Omega)$ is the  Hilbert space of square integrable scalar functions endowed with the norm $\lVert z \rVert=(\int_{\Omega}|z(x)|^2dx)^{\frac 12}$. We also define the Sobolev spaces $H^1_0(\Omega)=\{ z\in L^2(\Omega), \nabla z\in \left(L^2(\Omega)\right)^N, z=0 \text{ on } \partial \Omega \}$, equipped with the norm $\|z\|_{H^1_0(\Omega)}=\|\nabla z\|$ and $H^2(\Omega)=\{ (z, \nabla z)\in \left(L^2(\Omega)\right)^{N+1}, \partial_{x_j}\partial_{x_i} z\in L^2(\Omega) \}$, which is the set of all function such that $ \int_{ \Omega}\left( |z|^2+|\nabla z|^2+|\Delta z|^2  \right)dx$ is finite.
We will often write  $\int_{ \Omega}g(t)$ instead of $\int_{ \Omega}g(x,t)dx$ to ease the reading. 
Finally, a real symmetric positive definite matrix $M$ is denoted $M \succ 0$ and in  a partitioned matrix, the symbol $\star$ stands for symmetric blocks.	

\section{Problem formulation}\label{problemformuation}
Consider a damped wave equation
\begin{equation}
\label{wave}
\left\{
\begin{array}{ll}
\partial_t^2 z(x,t)-\Delta z(x,t)=-\alpha \partial_t z(x,t),  & \forall (x,t)\in \Omega\times\R^+,\\
z(x,t)=0, & \forall (x,t)\in \partial \Omega\times\R^+,\\
z(x,0)=z_0(x),\quad \partial_t z(x,0)=z_1(x),  &\forall  x\in \Omega,
\end{array}
\right.
\end{equation} where $\alpha>0$ is the damping coefficient and $\Omega$ is an open bounded domain in $\R^N$, with smooth boundary $\partial \Omega.$  

We are interested by an event-triggered implementation of the control term $-\alpha \partial_t z$, so that the control signal applied to the plant is updated only at certain instant $\{t_k\}_{k\in\N},$ defined by a mechanism. We assume that the control action is held constant between two successive events. Moreover, differently from classical periodic sampling techniques, the inter-sampling time $t_{k+1}-t_k$ is not assumed to be constant.
The closed-loop system can then be described as follows:
\begin{equation}
\label{waveclosedloop}
\left\{
\begin{array}{ll}
\partial_t^2 z-\Delta z=-\alpha\partial_t z(t_k),   & \hbox{in } \Omega\times [t_k,t_{k+1}), k\in \N\\
z=0, &  \hbox{on } \partial \Omega\times\R^+,\\
z(\cdot,0)=z_0,\partial_t z(\cdot,0)=z_1, & \hbox{in } \Omega.
\end{array}
\right.
\end{equation}
Note that $t_k, k=0,1,\cdots,$ are the triggering instants that satisfy $$0=t_0<t_1<\cdots<t_k<t_{k+1}< \cdots.$$
Hence, the problem we intend to solve can be summarized as follows.
\begin{problem}\label{prob1}
	Design a triggering condition in order to guarantee:
	\begin{enumerate}
		\item the well-posedness of the closed-loop system \eqref{waveclosedloop},
		\item the exponential stability of the system \eqref{waveclosedloop},
		\item the avoidance of Zeno behavior.	
	\end{enumerate}
\end{problem}

To address Problem \ref{prob1}, as a stepping stone, we exploit and expand the results about the continuous-time version of system \eqref{waveclosedloop}, that is, system \eqref{wave}. 

Indeed, the well-posedness and exponential stability of system \eqref{wave} have been widely studied in the literature. 
For instance, in \cite[Theorem 2.3 and Theorem 3.4]{chen1979control}, \cite{lions1988controlabilite} using the Hille-Yossida's theorem, the authors prove that for any initial conditions 
$(z_0,z_1)\in H^1_0(\Omega)\times L^2(\Omega)$,
there exists a unique weak solution to \eqref{wave} satisfying	 
$$
z\in C^0([0,T];H^1_0(\Omega))\cap C^1([0,T];L^2(\Omega)).
$$
Moreover it is proved, thanks to a multiplier technique, that the system is exponentially stable. More precisely, there exist $C>0$ and $\delta>0$ such that, for any initial condition in $H^1_0(\Omega)\times L^2(\Omega),$ the weak solution $z$ to \eqref{wave} satisfies, for all $t> 0$, 
$$
E(t)\le CE(0)e^{-\delta t},
$$
where the energy $E$  is defined as the sum of the kinetic and potential energies 
\begin{equation} 
\label{energy}
E(t)=\dfrac{1}{2}\left(\lVert \partial_t z(t)\rVert^2+\lVert \nabla z(t)\rVert^2\right).	
\end{equation}


\section{Event-triggering strategy}\label{etmstrategy}
In order to expand the event-triggering strategy developed in the context of finite-dimensional systems, as for example in \cite{tabuada2007event,postoyan2014framework,girard2014dynamic}, let us introduce the following error deviation between the speed at the last triggering instant and the current one, for all $x\in\Omega$ and $t\in[t_k,t_{k+1})$:
\begin{equation}
\label{ek}
e_k(x,t)=\partial_t z(x,t)-\partial_t z(x,t_k).
\end{equation}
From there, we can characterize the event-triggering rule we propose to study as:
\begin{equation}\label{tk}
t_{k+1}\!=\inf \Big\{ t\geq t_k,
\|e_k(t)\|^2> 2\gamma E(t)\Big\},
\end{equation}	where $\gamma>0$ is a design parameter.
The idea consists in measuring the deviation of the damping between the last sampled state and the current one and authorize it to be in a $\gamma$ proportion of the current energy.
In other words, between two triggering instants, $\|e_k(t)\|^2 \leq 2\gamma E(t)$ holds, and soon as this becomes false, an update event is generated. 
\begin{remark}
	The triggering rule \eqref{tk} is a static rule since it is based on the use of the state of the system and does not contain an internal dynamical variable as in \cite{girard2014dynamic,espitia2016event,espitia2017event}. This event-triggering rule is also different from that one considered in  \cite[Definition 2]{espitia2017stabilization}, \cite[Definition 3 ]{ESPITIA2016275}, \cite{kang2020event,baudouin2019event}  in the sense that no term built from the initial condition of the energy is added. 
\end{remark}  
Using \eqref{ek}, the closed-loop system under consideration can be written as follows: 
\begin{equation}
\label{waveclosedlooperror}
\left\{
\begin{array}{ll}
\partial_t^2 z-\Delta z=-\alpha\partial_t z+\alpha e_k,   & \hbox{ in } \Omega\times [t_k,t_{k+1}),\\
z=0, &  \hbox{ on } \partial \Omega\times\R^+,\\
z(\cdot,0)=z_0, 	\partial_t z(\cdot,0)=z_1, & \hbox{ in } \Omega.
\end{array}
\right.
\end{equation}

In the following we separate the study of the well-posedness of system \eqref{waveclosedlooperror}, from the exponential stability of the closed-loop system and the guarantee of the avoidance of Zeno behavior. 

\subsection{Well-posedness}\label{wellposedness}

Let us begin by defining the maximal time $T$ under which the system \eqref{waveclosedloop} subjected to the event-triggering law \eqref{tk} has a solution:
\begin{equation}
\label{T}
\left\{
\begin{array}{ll}
T=+\infty  &  \text{ if $(t_k)$ is a finite sequence, }\\
T=\displaystyle\limsup_{k\to +\infty} t_k &  \text{ if not}.
\end{array}
\right.
\end{equation}
From there, one can understand that later in the article, the proof of the absence of Zeno behavior will actually be stemming from the proof that $T=+\infty $ since no accumulation point of the sequence $(t_k)_{k\geq 0}$ will be possible.

In this section, leveraging on some regularity of the solutions to the wave equation we prove the following theorem.
\begin{theorem}  \label{thm1} 
	Let $\Omega$ be an open bounded domain of class $C^2$.  For any initial conditions  $(z_0,z_1)\in H^2(\Omega)\cap  H^1_0(\Omega)\times H_0^1(\Omega)$, there exists a unique strong solution to \eqref{waveclosedloop} under the event-triggering mechanism \eqref{tk},  satisfying 
	\begin{equation} \label{class}
	z\in C^0([0,T);H^2(\Omega)\cap H_0^1(\Omega))\cap C^1([0,T);H^1_0(\Omega)).
	\end{equation}	
\end{theorem}

\begin{proof}    
	\label{proofprop}	We will proceed verbatim as in \cite{baudouin2019event}. First of all, we show by induction the well-posedness on every sampled interval $[t_k,t_{k+1}]$. From the definition \eqref{T} of $T$, this will allow to obtain a unique solution in the class~\eqref{class}.
	
	\begin{itemize}
		\item \textbf{ Initialization.}
		On the first time interval $[0,t_1]$, the control system \eqref{waveclosedloop} reads 	
		\begin{equation}
		\label{waveclosedloop1}
		\left\{
		\begin{array}{ll}
		\partial_t^2 z-\Delta z=-\alpha z_1,   &  \text{ in } \Omega\times [0,t_{1}),\\
		z=0, &   \text{ on  } \partial \Omega\times(0,t_1),\\
		z(\cdot,0)=z_0, \partial_t z(\cdot,0)=z_1, &  \text{ in } \Omega,
		\end{array}
		\right.
		\end{equation}
		that is a simple wave equation with initial data  $(z_0,z_1)\in H^2(\Omega)\cap  H^1_0(\Omega)\times H_0^1(\Omega)$ and source term $f(t,x)=-\alpha z_1(x).$ Since $z_1\in H_0^1(\Omega)$, then $f\in L^1(0,t_1;H_0^1(\Omega)).$
		Thus from  \cite[Theorem 2.2]{chen1979control} it follows that there exists a unique solution satisfying 
		\begin{equation*}
		\label{waveregularity1}
		z\in C([0,t_1];H^2(\Omega)\cap H_0^1(\Omega))\cap C^1([0,t_1];H^1_0(\Omega)).
		\end{equation*}
		
		\item \textbf{Heredity.}
		Let $k\in\mathbb{N} $ be fixed and assume that 	
		\begin{equation*}
		\label{waveregularity2}
		z   \in C([t_k,t_{k+1}];H^2(\Omega)\cap H_0^1(\Omega))\cap C^1([t_k,t_{k+1}];H^1_0(\Omega)).
		\end{equation*}
		Consider now the closed-loop system \eqref{waveclosedloop} over the next time interval $[t_{k+1},t_{k+2}]$:
		\begin{equation*}
		\label{waveclosedloop3}
		\left\{
		\begin{array}{ll}
		\partial_t^2 z-\Delta z=-\alpha z_{2k+3}, & \text{ in } \Omega\times [t_{k+1},t_{k+2}],\\
		z=0, &  \hbox{ on } \partial \Omega\times[t_{k+1},t_{k+2}],\\
		z(\cdot,t_{k+1})=z_{2k+2}, & \text{ in }   \Omega,\\
		\partial_t z(\cdot,t_{k+1})=z_{2k+3}, & \text{ in } \Omega,
		\end{array}
		\right.
		\end{equation*}where we have denoted by $z_{2k+2}$ and $z_{2k+3}$ the position and velocity function values of the wave at $t_{k+1}$ given by the previous system over $[t_k,t_{k+1}]$. This is again a wave equation with source term which belongs to $ L^1([t_{k+1},t_{k+2}];L^2(\Omega))$ since we assumed $z\in C^1([t_{k},t_{k+1}];H^1_0(\Omega))$ and $\partial_t z(t_{k+1})=z_{2k+3}.$ Therefore, applying again \cite[Theorem 2.2]{chen1979control}  we conclude to the existence and the uniqueness of the solution $z$ in the same functional spaces on next time interval $[t_{k+1},t_{k+2}]$.
		
	\end{itemize}		
	By induction, this regularity holds for any $k\in\mathbb{N}$. 
	Therefore, from the extension by continuity at the update instants $t_k$, one can conclude that system \eqref{waveclosedloop}, or equivalently system \eqref{waveclosedlooperror}, has a unique solution in the class \eqref{class}.
	
	The fact that Theorem \ref{thm1} holds means that we solved item 1 of Problem \ref{prob1}. 
\end{proof}
\subsection{Exponential stability}\label{expostab}
In this section we address item 2 of Problem \ref{prob1}, that is, we propose sufficient conditions in order to ensure the exponential stability of system \eqref{waveclosedloop}-\eqref{tk}. The following result can be stated.
\begin{theorem} \label{mainteorem1} 
	
	Given  the damping parameter $\alpha>0$, assume there exist positive scalars  $\varepsilon,\gamma, \lambda_1,\lambda_2,\delta$ such that the following matrix inequality holds: 
	\begin{equation}\label{Phi}
	\Phi:=\begin{pmatrix}
	-\lambda_1+\alpha \varepsilon \delta&\delta\varepsilon&\frac{\alpha\varepsilon}{2}&0\\
	\star  &\phi_{22}&\frac\alpha{2}&0\\
	\star &\star &-\lambda_2& 0  \\
	\star &\star&\star& \phi_{44}
	\end{pmatrix}\prec 0
	\end{equation}
	with 
	\begin{equation*} \label{phi223} 
	\phi_{22}=\varepsilon-\alpha+\delta+\lambda_2 \gamma,\quad
	\phi_{44}=\delta-\varepsilon+\lambda_1C_{\Omega}^2+\lambda_2\gamma,
	\end{equation*}
	and $C_{\Omega}$ the constant in the Poincar\'e inequality (see Lemma \ref{poincare} in Appendix). Then,  for any initial condition $(z_0,z_1)\in H^2(\Omega)\cap  H^1_0(\Omega)\times H_0^1(\Omega),$
	the closed-loop system \eqref{waveclosedloop} or \eqref{waveclosedlooperror} under the event-triggering mechanism \eqref{tk} is exponentially stable with decay rate $\delta.$ In other words, there exists $ K>0$ such that 
	\begin{equation} \label{lmistab1}
	E(t)\le  K E(0)e^{-2\delta t}\quad \forall t>0.
	\end{equation} 
	Furthermore, if the above matrix inequality holds with $\delta=0,$ then the closed-loop system is exponentially stable with a small enough decay rate.	  	
\end{theorem}

\begin{proof}  \label{exponential}	Let be $\varepsilon>0$ and define the following Lyapunov functional candidate:
	\begin{multline*}
	V(t):=\dfrac{1}{2} \int_{\Omega} |\partial_t z(x,t)|^2dx+\dfrac{1}{2}\int_{\Omega} |\nabla z(x,t)|^2dx+\dfrac{\alpha\varepsilon}{2}\int_{\Omega}| z(x,t)|^2dx+\varepsilon\int_{\Omega}z(x,t)\partial_t z(x,t)dx.
	\end{multline*}
	To see the relationship between $V(t)$ and $E(t)$ note first that $E(t)\le V(t)$ and then by Cauchy-Schwarz, Young and Poincar\'e's inequalities (see Lemmas \ref{cauchy} and \ref{poincare} in Appendix), it follows:
	\begin{align*}
	V(t)&\le E(t)+\dfrac{\alpha\varepsilon}{2}\lVert z(t)\rVert^2+\varepsilon \lVert z(t)\rVert \lVert\partial_t z(t)\rVert\\
	&\le E(t)+\dfrac{\alpha\varepsilon}{2}\lVert z(t)\rVert^2+\dfrac{\varepsilon}{2}\left(C_\Omega  \lVert\partial_t z(t)\rVert^2  +\dfrac{1}{C_\Omega}\lVert z(t)\rVert^2\right)\\	
	&\le E(t)+\dfrac{\varepsilon}{2} C_\Omega  \lVert\partial_t z(t)\rVert^2 +\dfrac{C_\Omega^2 \varepsilon}{2}\left(\alpha+\dfrac{1}{C_\Omega}\right)\lVert \nabla z(t)\rVert^2,  
	\end{align*}
	or equivalently
	$$V(t)	\le \left( 1+\varepsilon C_\Omega+\varepsilon\alpha C_\Omega^2  \right) E(t)=C_r E(t).	
	$$
	Hence we have \begin{equation}\label{equivalenceofv}
	E(t)\le V(t)\le C_r E(t).
	\end{equation}	
	Moreover, we want to ensure that: 
	$$
	\dot{V}(t)+2\delta V(t)\le 0, \qquad \forall t\ge 0	.
	$$
	Thus let us start by computing the time-derivative of $V$ along the trajectories of \eqref{waveclosedloop}:	
	\begin{align} \label{vprime}
	\dot{V}(t)=&~ \dot{E}(t)+\alpha\varepsilon\int_{\Omega} z(t)\partial_t z(t)
	+\varepsilon\int_{\Omega}|\partial_t z(t)|^2+\varepsilon\int_{\Omega}z(t)\partial_t^2 z(t).
	\end{align}
	From now on, the dependence of the intervals in terms of the mute variable $x$ is ghosted in order to ease the reading.
	Let us recall that from \eqref{waveclosedlooperror} one has for all $(x,t)$ in $\Omega\times \R_+$:
	\begin{equation} \label{eqwave}
	\partial_t^2 z(x,t)=\Delta z(x,t)-\alpha\partial_t z(x,t)+\alpha e_k(x,t).
	\end{equation} 
	In addition to this, one needs to express the time-derivative of the energy $\dot{E}(t).$
	Then one gets
	\begin{equation*}\label{en1}
	\dot{E}(t)=\displaystyle\int_{\Omega}\partial_t^2 z(t)\partial_t z(t)+\displaystyle\int_{\Omega}\nabla\partial_t z(t) \nabla  z(t),
	\end{equation*}
	which gives by using \eqref{eqwave} and the Green formula:
	\begin{align}\label{energiereducted}
	\dot{E}(t)&=-\alpha\displaystyle\int_{\Omega}|\partial_t z(t)|^2+\alpha\displaystyle\int_{\Omega}e_k(t)\partial_t z(t).
	\end{align}
	Gathering \eqref{eqwave}, \eqref{energiereducted} 	and  \eqref{vprime} we obtain:	 
	\begin{align*} 
	\dot{V}(t)+2\delta V(t) =~&\alpha\varepsilon\delta\int_{\Omega}|z(t)|^2
	+(\delta-\varepsilon)\int_{\Omega} |\nabla z(t)|^2  +\alpha\varepsilon\int_{\Omega} z(t)e_k(t)
	+\alpha\int_{\Omega} \partial_t z(t)e_k(t)\\
	&+(\varepsilon-\alpha+\delta )\int_{\Omega} |\partial_tz(t)|^2
	+2\delta\varepsilon\int_{\Omega}z(t)\partial_t z(t),
	\end{align*}
	so that 
	\begin{equation}\label{VVV}
	\dot{V}(t)+2\delta V(t) =\displaystyle\int_{\Omega} \psi^{\top}(x,t)M_1\psi(x,t) dx,
	\end{equation}
	with $\psi=\begin{pmatrix}z&\partial_t z &e_k& \nabla z \end{pmatrix}^{\top}$ and a symmetric matrix \\
	$$M_1=\begin{pmatrix}
	\alpha \varepsilon \delta&\delta \varepsilon&\frac{\alpha\varepsilon}{2}&0\\
	\star  &\varepsilon-\alpha+\delta &\frac\alpha{2}&0\\
	\star &\star& 0 & 0\\
	\star &\star&\star&\delta-\varepsilon
	\end{pmatrix}.$$
	Actually we want to satisfy $\dot{V}(t)+2\delta V(t)\le 0$ or equivalently $	\int_{\Omega} \psi^{\top}(t)M_1\psi(t) \le 0$ subject to some constraints. 
	
	The first constraint comes from the	Poincar\'e's inequality $\|z(t) \|^2\le C_{\Omega}^2 \|\nabla z(t)\|^2$ (see Lemma \ref{poincare} in Appendix) and it is equivalent to  
	$$\displaystyle\int_{\Omega} \psi^{\top}(t)M_2\psi(t) \ge 0, \hbox{ with } M_2=diag (-1,0,0, C_{\Omega}^2).$$
	The second constraint comes from the event-triggering law that imposes $\|e_k(t)\|^2 \le 2\gamma E(t),\, \forall t\in [t_k,t_{k+1}),$ i.e., while no triggering event occurs. This last inequality can be written $\|e_k(t)\|^2 \le \gamma \left( \|\partial_t z(t)\|^2+\|\nabla z(t)\|^2 \right)$ or equivalently 
	$$\displaystyle\int_{\Omega}\psi(t)^{\top}M_3\psi(t) \ge 0 , \hbox{ with } M_3=diag (0,\gamma,-1,\gamma).$$
	Using the S-procedure \cite[Section 2.6.3]{boyd1994linear}, one therefore wants to satisfy the following condition:
	\begin{align}\label{sprocedure}
	\dot{V}(t)+2\delta V(t)+\lambda_1\int_{\Omega} \psi(t)^{\top} M_2 \psi(t)+\lambda_2\int_{\Omega} \psi(t)^{\top} M_3 \psi(t)\le 0	
	\end{align} 
	for any  two positive scalars $\lambda_1$ and $\lambda_2$.\\
	Since one has \eqref{VVV} then it reads:	
	\begin{equation}\label{sproreduit}
	\int_{\Omega} \psi^{\top}(x,t)(M_1+\lambda_1M_2+\lambda_2M_3)\psi(x,t) dx\le 0,
	\end{equation} 	Hence, by defining $\Phi$ as $\Phi=M_1+\lambda_1M_2+\lambda_2M_3$, the satisfaction of relation \eqref{Phi} means that relation \eqref{sproreduit} and \eqref{sprocedure} are also satisfied, and therefore one obtains $$\dot{V}(t)+2\delta V(t)\le 0,\quad \forall t\in \R^+.$$ 
	That corresponds	 to have 
	$ V(t)\le e^{-2\delta t}V(0).$ By taking \eqref{equivalenceofv} into account, it follows that $E(t)\le  C_rE(0) e^{-2\delta t}.$ The proof of Theorem\,\ref{mainteorem1} is complete. 
\end{proof}
Before studying the avoidance of Zeno behavior in the next section, let us provide some insights regarding the matrix inequality \eqref{Phi}. First, we can use a change of variable $\bar{\gamma}=\lambda_2\gamma$ and search both $\lambda_2$ and $\bar{\gamma}$ as decision variables of $\Phi\prec 0$ or more precisely of:
\begin{equation}\label{Phi1}
\Phi:=\begin{pmatrix}
-\lambda_1+\alpha \varepsilon \delta&\delta\varepsilon&\frac{\alpha\varepsilon}{2}&0\\
\star  &\varepsilon-\alpha+\delta+\bar{\gamma}&\frac\alpha{2}&0\\
\star &\star&  -\lambda_2 & 0\\
\star &\star&\star&\phi_{44}
\end{pmatrix}\prec 0
\end{equation} with $$\phi_{44}=\delta-\varepsilon+\lambda_1 C_{\Omega}^2+\bar{\gamma}.$$
The negativity of $\Phi$ is a sufficient condition allowing to ensure the exponential stability of the closed loop. In the following proposition we show that there always exists a solution $(\lambda_1,\lambda_2,\bar{\gamma},\delta)$ such that \eqref{Phi1} is satisfied.

\begin{proposition}\label{prop}
	Given $\alpha>0,$ condition \eqref{Phi} of Theorem\,\ref{mainteorem1} or equivalently condition \eqref{Phi1} enjoys the following properties
	\begin{enumerate}
		\item[(i)] Given $\delta=0,$ condition \eqref{Phi1} is always feasible;
		\item[(ii)] There always exists a positive scalar $\delta\neq 0$ such that \eqref{Phi1} is feasible.
	\end{enumerate}	
\end{proposition}
\begin{proof}
	Let us denote by $M_0$ the matrix corresponding to $-\Phi$ in the case $\delta= 0$:
	\begin{equation*}
	M_0:=\begin{pmatrix}
	\lambda_1&0&-\frac{\alpha\varepsilon}{2}&0\\
	0  &-\varepsilon+\alpha-\bar{\gamma}&-\frac\alpha{2}&0\\
	-\frac{\alpha\varepsilon}{2} &-\frac\alpha{2}&\lambda_2&0\\
	0 &0 & 0 &  \varepsilon-\lambda_1 C_{\Omega}^2-\bar{\gamma}\\
	\end{pmatrix}.
	\end{equation*}	
	Condition \eqref{Phi1} reads $M_0\succ 0,$ which corresponds to:
	\begin{align}
	\label{eq1}	\lambda_1>0,\lambda_2>0\\
	\label{eq2}	\alpha-\varepsilon-\bar{\gamma}>0&\Longleftrightarrow \varepsilon<\alpha-\bar{\gamma}\\
	\label{eq3}	\varepsilon-\lambda_1 C_{\Omega}^2-\bar{\gamma}>0&\Longleftrightarrow \varepsilon>\lambda_1 C_{\Omega}^2+\bar{\gamma}
	\end{align}
	and by using the Schur complement \cite{boyd1994linear} on the first $3\times 3$ block we get:
	\begin{equation}
	\label{shur}
	\begin{pmatrix}
	\lambda_1&0\\
	0  &-\varepsilon+\alpha-\bar{\gamma}
	\end{pmatrix}-\dfrac{\alpha^2}{4\lambda_2}\begin{pmatrix}
	\varepsilon\\1
	\end{pmatrix}\begin{pmatrix}
	\varepsilon &1
	\end{pmatrix}>0
	\end{equation}	
	By appling the elimination lemma to \eqref{shur} one gets that it is equivalent to 
	\begin{equation*}
	\begin{pmatrix}
	1&-\varepsilon
	\end{pmatrix}\begin{pmatrix}
	\lambda_1&0\\
	0  &-\varepsilon+\alpha-\bar{\gamma}
	\end{pmatrix}\begin{pmatrix}
	1\\-\varepsilon
	\end{pmatrix}>0
	\end{equation*} which is always satisfied provided that $\lambda_1,\varepsilon$ and $\bar{\gamma}$ satisfy \eqref{eq1} and \eqref{eq2}. Therefore one can conclude that if one can choose $\lambda_1,\varepsilon,\bar{\gamma}$ satisfying \eqref{eq1}, \eqref{eq2} and \eqref{eq3} there always exists $\lambda_2>0$ such that \eqref{shur} holds. In summary there exists a solution  such that  $M_0\succ 0$ holds.  
	Furthermore, the tuning parameter $\gamma$ is easily recovered from $\bar{\gamma}$ and~$\lambda_2$.
	
	Consider now $\delta\neq 0,$ then one can write $\Phi$ as follows:
	\begin{equation}\label{phiwritten}
	\Phi=-M_0+\delta \begin{pmatrix}
	\alpha\varepsilon& \varepsilon&0& 0\\
	\varepsilon  & 1 &0& 0\\
	0 &0& 1 & 0\\
	0 &0& 0 & 0\\
	\end{pmatrix}. 
	\end{equation}
	Since there exist $\varepsilon,\bar{\gamma},\lambda_1,\lambda_2$ such that $-M_0\prec 0,$ it follows that there always exists $\delta>0$ such that $\Phi\prec 0.$\\
	The proof of Proposition \ref{prop} is complete. 
\end{proof}

\subsection{Avoidance of Zeno behavior}\label{avoidancezeno}
In this section, we address the third item of Problem \ref{prob1}, namely we prove that we avoid Zeno behavior.

Before proving that this phenomenon cannot occur, let us show that the natural energy \eqref{energy} of the closep-loop system has a useful property stated as follows.

\begin{lemma}
	\label{lemmaenergy}
	Under the event-triggering mechanism \eqref{tk} there exists a constant $C>0$ such that for all $t\in [0,T)$ :  
	\begin{equation}
	\label{energy11}
	E(0)e^{-2Ct}\le E(t)\le E(0)e^{2Ct}.
	\end{equation} 
\end{lemma}

\begin{proof}
	Let us recall that the time-derivative of $E(t)$ used in the proof of Theorem \ref{mainteorem1} satisfies \eqref{energiereducted} so that one gets 
	\begin{align}
	\label{estimateEpoint} |\dot{E}(t)|\le & \alpha \|\partial_t z(t)\|^2+\alpha \|e_k(t)\| \|\partial_t z(t)\|
	\end{align}
	From the definition \eqref{T} of $T$, since $t\in[0,T)$, either there exists $k$ such that $ t\in [t_k,t_{k+1}) $ if the sequence $(t_k)_{k\geq 0}$ is not finite, or $t$ may be greater than the last $t_k$ and the definition of  \eqref{tk} allows to call $t_{k+1} = T$.\\
	Using the event-triggering law \eqref{tk} one gets for $t\in [t_k,t_{k+1}):$
	$$
	\|e_k(t)\|^2\le 2\gamma E(t)
	$$
	and since $\|\partial_t z(t)\|^2\le 2E(t)$ we get :
	\begin{align}\label{eprime}
	\nonumber	|\dot{E}(t)|&\le 2\alpha E(t)+\alpha\sqrt{2\gamma E(t)}\sqrt{2E(t)}\\
	\nonumber|\dot{E}(t)|	&\le 2 \alpha E(t)+2\alpha\sqrt{\gamma} E(t)  \text{ or equivalently  }\\
	|\dot{E}(t)|	&\le 2C E(t) \text{ with  $C=\alpha (1+\sqrt{\gamma})$}.	
	\end{align}
	It follows that $-2CE(t)\le \dot{E}(t)\le 2CE(t)$.\\  
	Gronwall's Lemma applied on $[t_k,t]$ (Lemma \ref{gronwall} in Appendix) to both inequalities gives 
	\begin{equation}\label{estimeenergy}	
	E(t_k)e^{-2C(t-t_k)}\le E(t)\le E(t_k)e^{2C(t-t_k)}. 
	\end{equation}	
	Then taking $t=t_{k+1},$ it becomes : 
	\begin{equation*}\label{Eentkmoins}
	E(t_k)e^{-2C(t_{k+1}-t_k)}\le E(t_{k+1})\le E(t_k)e^{2C(t_{k+1}-t_k)}.
	\end{equation*}
	Inferring what it gives for $E(t_k)$,  one can deduce
	\begin{equation*} \label{Eentkmoins1}
	E(t_{k-1})e^{-2C(t_{k+1}-t_{k-1})}\le E(t_{k+1})\le E(t_{k-1})e^{2C(t_{k+1}-t_{k-1})}
	\end{equation*}
	and since $t_0=0,$ by induction we get:
	\begin{align*}
	E(0)e^{-2Ct_{k+1}}&\le E(t_{k+1})\le E(0)e^{2Ct_{k+1}}. 
	\end{align*}
	Then inequality \eqref{estimeenergy} yields:
	\begin{align*}
	E(0)e^{-2Ct_k}e^{-2C(t-t_k)}\le E(t)\le E(0)e^{2Ct_k}e^{2C(t-t_k)} ,
	\end{align*} showing that \eqref{energy11} holds for all $t\in [0,T)$.{\hfill $\Box$ } 
\end{proof}

We can now state the following result about the avoidance of Zeno behavior.
The idea is to consider the maximal time $T$ under which we proved that the system \eqref{waveclosedloop} subjected to the event-triggered law \eqref{tk} has a solution. 
From the definition of $T$ in \eqref{T}, one can verify that
if $T<+\infty,$ then $T$ is an accumulation point of the sequence $(t_k)_{k\ge 0}$ and a Zeno phenomenon occurs. Thus, avoiding Zeno phenomenon is a consequence of proving that $T=+\infty.$
\begin{theorem} \label{zeno}
	There is no Zeno Phenomenon for the system \eqref{waveclosedlooperror} under the event-triggering mechanism \eqref{tk}. 
	Equivalently, the maximal time defined by \eqref{T} is actually $T=+\infty.$
\end{theorem}

\begin{proof} By taking inspiration from the  reasoning as in \cite{tabuada2007event,girard2014dynamic} the proof is based on the study of the function $\varphi$  defined on $[t_k,t_{k+1})$ by 
	\begin{equation}
	\label{phi}
	\varphi: t\mapsto \varphi(t)=\frac{ \|e_k(t)\|^2}{2 \gamma E(t) }.
	\end{equation}	
	Let us estimate the time-derivative of $\varphi:$
	\begin{equation} \label{phiprime}
	\dot{\varphi}(t)=\frac{\displaystyle\int_{\Omega}\dot{e}_k(t)e_k(t) }{ \gamma E(t)}-\frac{\dot{E}(t)\|e_k(t) \|^2 }{ 2\gamma\left(  E(t)\right)^2 }.
	\end{equation}
	We have from \eqref{ek}, \eqref{waveclosedlooperror} and the Cauchy Schwartz's inequality, $\forall t\in [t_k,t_{k+1});$ 
	\begin{align*}
	\int_{\Omega}&\dot{e}_k(t)e_k(t)\\
	&=\int_{\Omega}\Delta z(t)e_k(t)-\alpha\int_{\Omega}\partial_tz(t)e_k(t)+\alpha\|e_k\|^2,\\
	&\le \|e_k(t)\|\|\Delta z(t)\|+\alpha \|e_k(t)\|\|\partial_t z(t)\|+\alpha \|e_k(t)\|^2.
	\end{align*}
	Since for any $(z_0,z_1)\in H^2(\Omega)\cap  H^1_0(\Omega)\times H_0^1(\Omega)$, the closed-loop system \eqref{waveclosedlooperror} with  \eqref{tk} has a unique solution satisfying 
	$z\in C^0([0,T);H^2(\Omega)\cap H_0^1(\Omega))$, then there exists a constant $C_1>0$ such that $\forall t\in [0,T)$
	\begin{equation}
	\label{placien}
	\|\Delta z(t)\|\le	\|\Delta z(t)\|_{L^{\infty}(0,T;L^2(\Omega))}\le C_1,
	\end{equation}
	where $C_1$ depends on $\|z_0\|_{H^2(\Omega)}$ and $\|z_1\|_{H^1_0(\Omega)}$. \\
	Then  using $\|\partial_tz(t)\|^2\le 2 E(t)$ and \eqref{tk} it follows: 
	\begin{align*}
	\dfrac{\displaystyle\int_{\Omega}\dot{e}_k(t)e_k(t)}{\gamma E(t)}
	&\le \frac{\|e_k(t)\|\|\Delta z(t)\|}{\gamma E(t)}+\dfrac{\alpha \|e_k(t)\|\|\partial_t z(t)\|}{\gamma E(t)} +\alpha\frac{ \|e_k(t)\|^2}{ \gamma E(t) }\\
	&\le \frac{C_1\sqrt{2\gamma E(t)}}{\gamma E(t)}+\dfrac{\alpha\sqrt{2\gamma E(t)}\sqrt{2 E(t)}}{\gamma E(t)} +\alpha\varphi(t),
	\end{align*}	
	which leads to: 
	\begin{align}\label{term3}
	\dfrac{\displaystyle\int_{\Omega}\dot{e}_k(t)e_k(t)}{\gamma E(t)}	\le \frac{C_1\sqrt{2}}{\sqrt{\gamma E(t)}}+\dfrac{2\alpha}{\sqrt{\gamma}} +\alpha\varphi(t).
	\end{align}
	Using \eqref{eprime} we get:\begin{equation}
	\label{term1}
	\dfrac{	-\dot{E}(t)\|e_k(t) \|^2 }{2\gamma \left(E(t)\right)^2}\le \alpha(1+\sqrt{\gamma})\varphi(t).
	\end{equation}
	Gathering the terms \eqref{term3} and \eqref{term1}  we have:
	$$
		\dot{\varphi}(t)\le  \frac{C_1\sqrt{2}}{\sqrt{\gamma E(t)}}+\dfrac{2\alpha}{\sqrt{\gamma}}  +\alpha(2+\sqrt{\gamma})\varphi(t)
		\le \frac{C_1\sqrt{2}}{\sqrt{\gamma E(t)}}+\dfrac{2\alpha}{\sqrt{\gamma}} +\alpha (2 +\sqrt{\gamma}).
	$$
	Recall that from the event-triggering law \eqref{tk}, an event occurs if $\varphi(t)> 1$, and as long as $\varphi(t)\le 1$, no update event is trigerred.  Hence it follows:
	\begin{equation}\label{phiprim}
	\dot{\varphi}(t)		\le A+\frac{B}{\sqrt{E(t)}}
	\end{equation}
	with $A=\dfrac{2\alpha}{\sqrt{\gamma}} +\alpha(2+\sqrt{\gamma})$ and $B=C_1\sqrt{\dfrac{2}{\gamma}}.$
	
	Using Lemma \ref{lemmaenergy}  one gets 
	$$\forall t\in [0,T),  E(t)\ge E(0)e^{-2Ct}\ge E(0)e^{-2CT},$$ and \eqref{phiprim} becomes $\dot{\varphi}(t)\le  A+\frac{Be^{CT}}{\sqrt{E(0)}} .$ Then $\forall k\in \N,$  integrating on $[t_k,t_{k+1}]$ knowing that $\varphi(t_k)=0$ and $\varphi(t_{k+1})=1$ we obtain : \begin{equation}
	\label{tkk}
	1\le \left[A+\frac{Be^{CT}}{\sqrt{E(0)}}\right](	t_{k+1}-t_k).
	\end{equation}
	Let $\displaystyle t_k\to T$ as $k\to +\infty$ in \eqref{tkk}, then we get a contradiction if $T\neq +\infty$ and then we need to consider $T = +\infty$. That leads to the absence of any accumulation points. 
	Therefore, the avoidance of Zeno behavior is guaranteed. 
\end{proof}
%
\subsection{Main result}
The main result to solve Problem \ref{prob1} is obtained by combining Theorems \ref{thm1}, \ref{mainteorem1} and \ref{zeno}. This is summarized below.
\begin{theorem}
	Let $\alpha>0$ be a damping parameter.  For any initial conditions  $(z_0,z_1)\in H^2(\Omega)\cap  H^1_0(\Omega)\times H_0^1(\Omega)$, 
	system \eqref{waveclosedloop} under the event-triggering rule \eqref{tk} satisfies the following properties:
	\begin{enumerate}
		\item There exists a unique strong solution satisfying \eqref{class}.
		
		\item The closed-loop exponential stability is ensured provided that the relation \eqref{Phi} is feasible.
		
		\item The absence of accumulation points on the sequence of triggering instants ensures the avoidance of Zeno behavior.
		
	\end{enumerate} 
\end{theorem}

\section{Numerical simulation}\label{simulation}
We illustrate the efficiency of the event-triggering law proposed in this paper by considering the example of a one dimentional wave equation on $\Omega=(0,\pi).$

We compare the behavior of the continuous-time version of the closed-loop system versus the event-triggered closed-loop version. In other words we compare the behavior of system \eqref{wave} with the one of system \eqref{waveclosedloop}  under the event-triggering rule \eqref{tk}. To do this, let us consider the initial conditions
\begin{equation}\label{initial}
z_0(x)=\sin \left(x\right) \text{ and } z_1(x)=\sin \left(2 x\right).
\end{equation}
and the damping coefficient $\alpha=1.$

Figure\,\ref{continuous} depicts the numerical solution of the closed-loop system \eqref{wave}.

\begin{figure}[h!]
$$	\includegraphics[scale=0.45]{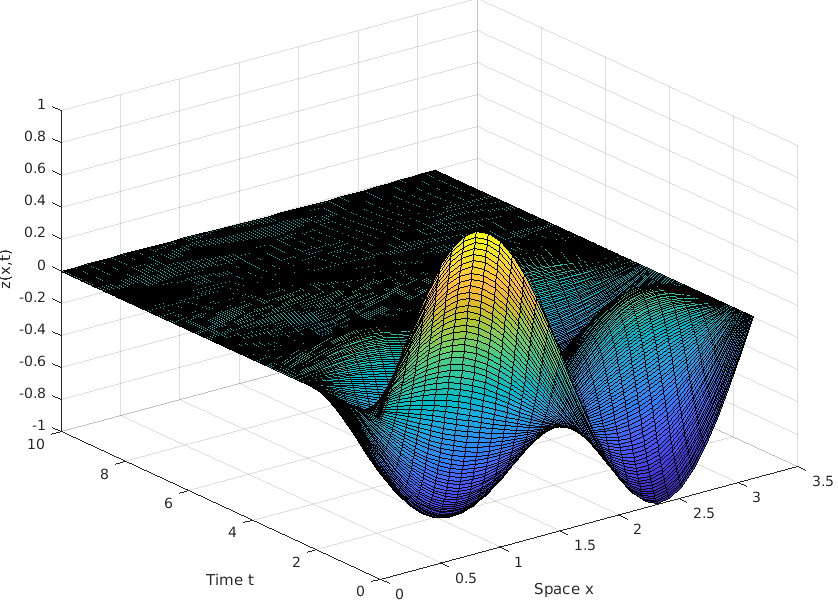} $$
	\caption{ Numerical solution to the closed-loop system \eqref{wave} under the continuous time controller with the damping coefficient $\alpha=1$ for the initial condition \eqref{initial}.}
	\label{continuous}
\end{figure} 

The design parameter $\gamma$ in the event-triggering rule \eqref{tk} plays a key role in the exponential stability of system~\eqref{waveclosedloop} under the mechanism \eqref{tk}.
The choice of  $\gamma$ influences the number of updates imposed by \eqref{tk}: the smaller $\gamma$, the more frequent the updates.
A feasible solution to condition \eqref{Phi} in Theorem \ref{mainteorem1} 
is:  $\lambda_1=0.1, \lambda_2=1, \gamma=0.02, \delta=0.25$ and $\varepsilon=0.8$.
As depicted on Figure~\ref{event-triggered}, the solution to system \eqref{waveclosedloop}-\eqref{tk} converges quickly to the origin. 
\begin{figure}[h!]
$$	\includegraphics[scale=0.50]{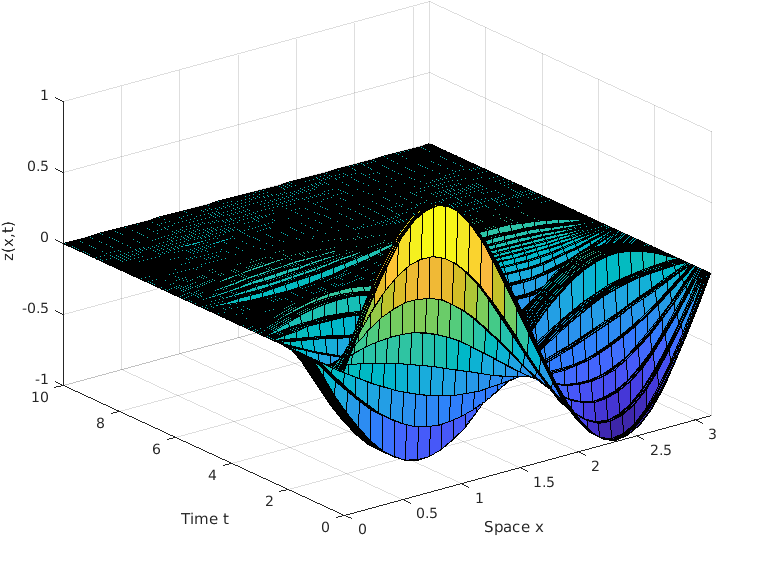} $$
	\caption{ Numerical solution to the closed-loop system \eqref{waveclosedloop} under the event-triggering mechanism \eqref{tk} with the damping coefficient $\alpha=1$ for the initial condition \eqref{initial}.}
	\label{event-triggered}	
\end{figure} 

Now in other to better understand how the sampling acts on the exponential stability result, we present in Figure~\ref{energies}, the evolution of the natural energy 
$$E(t)=\dfrac{1}{2} \displaystyle\int_{0}^\pi |\partial_t z(x,t)|^2dx+\dfrac{1}{2} \displaystyle\int_{0}^\pi |\partial_x z(x,t)|^2dx$$ of the closed-loop system \eqref{waveclosedloop} in the following cases (recall that one considers $\alpha=1$):
\begin{itemize}
	\item Under the continuous controller (blue line)  
	\item Under the event-triggered controller (black dotted line) with $t_k$ given by the event-triggering rule \eqref{tk}.
	\item Under the fixed controller 	$f(x,t)=-\partial_t z(x,t_0)=- z_1(x)$ (green line).  	
	\item  With the controller $f(x,t)=-\partial_t z(x,kp)$ (red dashed line) under periodic sampling with period 
	$\tau=\frac{T}{N_{\text{up}}}$ where $N_{\text{up}}$ is the number of update observed during the time $T$ when following \eqref{tk}. 
\end{itemize}
\begin{center}
	\begin{figure}[htb!]
	$$	\includegraphics[scale=0.7]{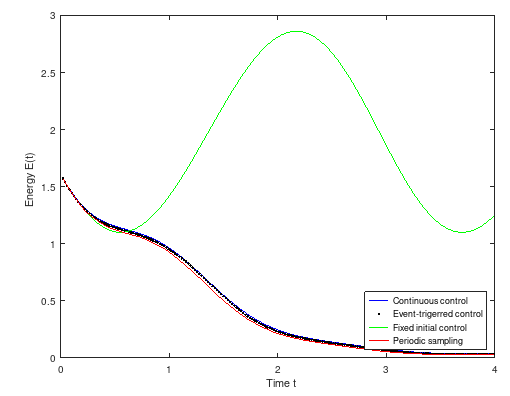} $$
		\caption{ Evolution of the energy $E(t)$ given by \eqref{energy}}
		\label{energies}	
	\end{figure}
\end{center}

First, we remark that when the controller is fixed as $f(t)=- z_1,$ the energy evolves as the sinusoidal $z_1.$ This corresponds to the first instant of sampling where no event occurred before. 
Second, Figure\,\ref{energies} shows that the evolution of energy of the event-triggered control system is similar to the one of the continuous-time controlled system and to the one under ad-hoc periodic sampling. Nevertheless, to the best of our knowledge a proof of the exponential decay of the energy in the case of periodic sampling does not exist. But as it can be seen, a good choice of the period $\tau$ leads to the exponential decay of the corresponding energy. More precisely, using trial and error method, one can find that this system becomes unstable when $\tau>\frac{T}{N_{\text{up}}}.$

Finally, Figure \ref{control} depicts the evolution of the magnitude of the controller $\|f(t)\|_{L^2(0,\pi)}$ in the continuous-time and the event-triggering frameworks. We notice that the update times are not regular and there is a large variation in the magnitude of the continuous-time controller allowing to conclude that the event-triggered control approach is energy efficient.

\begin{center}
	
	\begin{figure}[htb!]
	$$	\includegraphics[scale=0.45]{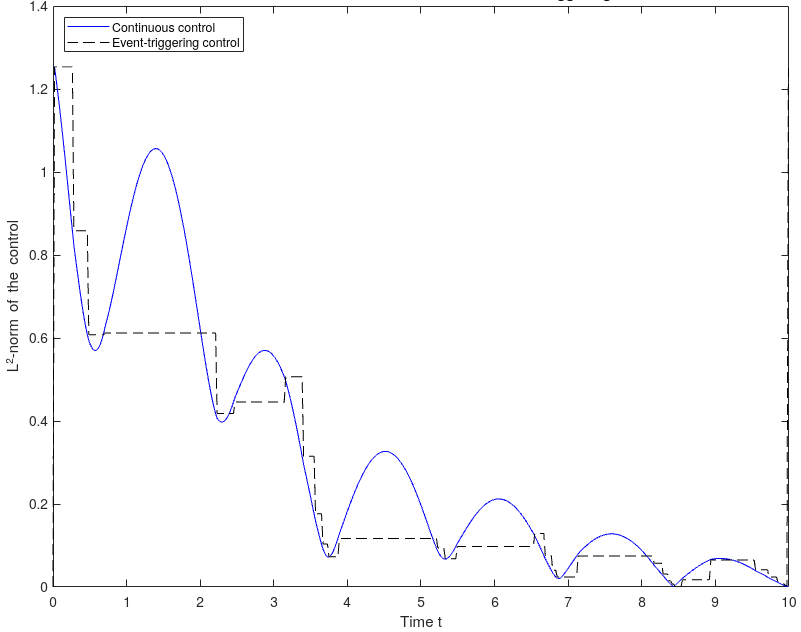} $$
		\caption{ Evolution of the $L^2$-norm of the continuous-time controller  in blue and the event-triggered controller  (in black dashed line) with damping coefficient $\alpha=1.$ }
		\label{control}
	\end{figure} 
\end{center}

\section{Conclusion}\label{conclusion}
In the present paper, the exponential stabilization of the damped linear wave equation under an event-triggering mechanism is guaranteed. A sufficient matrix inequality-based condition for the
exponential stability of the system is formulated. The avoidance of the Zeno behavior through the absence of accumulation points in the updates sequence was proved. 

This work open the door for future investigations. In particular, it would be relevant to study other classes of PDEs, for example those appearing in vibration control theory as the beam equation \cite{crepeau2006control}. Moreover, the system could also be subjected to input nonlinearity \cite{prieur2016wave} and this is also another interesting direction of future works.

\textbf{ Acknowledgement}	The authors would like to thank  Sylvain Ervedoza for interesting and fruitful discussions on the proof of the avoidance of Zeno phenomenon.
	

\appendix
\section{Useful Lemmas}
\begin{lemma}[\textnormal{Cauchy-Schwarz's inequality \cite{evans1998partial}}]\label{cauchy}
	For any  $u,v\in L^2(\Omega)$ it holds
	$$	\int_{\Omega} u(x)v(x)dx\le \lVert u\rVert_{L^2(\Omega)}\lVert v\rVert_{L^2(\Omega)}.$$
	
\end{lemma}

\begin{lemma}
	[\textnormal{Poincar\'e's inequality \cite{evans1998partial}}] \label{poincare}
	Let $\Omega$ be a bounded, connected, open subset of  $~\R^N ,$ of class $C^1 $. There exists a constant $C_\Omega$, depending only on $N$ and  on the diameter of the domain $\Omega,$ such that for each function $z\in H_0^1(\Omega)$,
	$$
	\left\| z\right\|_{L^2(\Omega)}\le C_\Omega\|	\nabla  z\rVert_{L^2(\Omega)}.
	$$
\end{lemma}

\begin{lemma}
	[\textnormal{Gronwall's Inequality \cite{ames1997inequalities}}]\label{gronwall}
	Let $u$ be a real-valued continuous function defined on an interval of the form $[a,\infty)$ or $[a,b]$ or $[a,b)$ with $a<b$. If $u$ is differentiable in $]a,b [$ and satisfies the differential inequality 
	$$ \dot{u}(t)\le \beta(t)u(t),\quad  \forall t\in  ]a,b [$$
	then 
	$$u(t)\le u(a)\exp \left(\int_{a}^{t}\beta(s)ds\right),\quad \forall t\in [a,b]. $$
\end{lemma}

	\bibliography{wave}
	\bibliographystyle{plain}	
\end{document}